\documentclass[11pt,a4paper,reqno]{amsart}
\usepackage[english]{babel}
\usepackage[T1]{fontenc}
\usepackage{verbatim}
\usepackage{palatino}
\usepackage{amsmath}
\usepackage{mathabx}
\usepackage{amssymb}
\usepackage{amsthm}
\usepackage{amsfonts}
\usepackage{graphicx}
\usepackage{esint}
\usepackage{color}
\usepackage{mathtools}
\usepackage{overpic}

\usepackage[colorlinks = true, citecolor = black]{hyperref}
\author{Tuomas Orponen}
\address{Department of Mathematics and Statistics\\ University of Jyv\"askyl\"a,
	P.O. Box 35 (MaD)\\
	FI-40014 University of Jyv\"askyl\"a\\
	Finland} 
\email{\href{mailto:tuomas.t.orponen@jyu.fi}{tuomas.t.orponen@jyu.fi}}

\title{A note on higher integrability of projections}
\date{\today}
\thanks{T.O. is supported by the European Research Council (ERC) under the European Union’s Horizon Europe research and innovation programme (grant agreement No 101087499), and by the Research Council of Finland via the project \emph{Approximate incidence geometry}, grant no. 355453.}

\newcommand{\R}{\mathbb{R}}

\newcommand{\N}{\mathbb{N}}

\newcommand{\Hd}{\dim_{\mathrm{H}}}

\newcommand{\diam}{\operatorname{diam}}

\newcommand{\dist}{\operatorname{dist}}

\def\Barint_#1{\mathchoice
          {\mathop{\vrule width 6pt height 3 pt depth -2.5pt
                  \kern -8pt \intop}\nolimits_{#1}}%
          {\mathop{\vrule width 5pt height 3 pt depth -2.6pt
                  \kern -6pt \intop}\nolimits_{#1}}%
          {\mathop{\vrule width 5pt height 3 pt depth -2.6pt
                  \kern -6pt \intop}\nolimits_{#1}}%
          {\mathop{\vrule width 5pt height 3 pt depth -2.6pt
                  \kern -6pt \intop}\nolimits_{#1}}}

\numberwithin{equation}{section}

\theoremstyle{plain}
\newtheorem{thm}{Theorem}
\numberwithin{thm}{section}
\newtheorem*{"thm"}{"Theorem"}

\newtheorem{lemma}[thm]{Lemma}

\newtheorem{proposition}[thm]{Proposition}

\newtheorem{claim}[thm]{Claim}

\theoremstyle{definition}

\newtheorem{definition}[thm]{Definition}

\theoremstyle{remark}
\newtheorem{remark}[thm]{Remark}

\newcommand{\nref}[1]{(\hyperref[#1]{#1})}

\DeclareMathSymbol{\intop}  {\mathop}{mathx}{"B3}

\begin{document}

\begin{abstract} Let $t \in [1,2)$ and $p > 2/(2 - t)$. I construct a $t$-Frostman Borel measure $\mu$ on $[0,1]^{2}$ such that $\pi_{\theta}\mu \notin L^{p}$ for every $\theta \in S^{1}$. This answers a question of Peres and Schlag. \end{abstract}

\maketitle


\section{Introduction}

A Borel measure $\mu$ on $\R^{d}$ is called \emph{$t$-Frostman} if there exists a constant $C > 0$ such that $\mu(B(x,r)) \leq Cr^{t}$ for all $x \in \R^{d}$ and $r > 0$. For $\theta \in S^{1}$, the orthogonal projection $\R^{2} \to \mathrm{span}(\theta^{\perp})$ is denoted $\pi_{\theta}$.

Marstrand \cite{Mar} in 1954 showed that if $\mu$ is a compactly supported $t$-Frostman measure on $\R^{2}$ with $t > 1$, then $\pi_{\theta}\mu$ is absolutely continuous with density in $L^{2}(\R)$ for $\mathcal{H}^{1}$ almost every $\theta \in S^{1}$. Can $L^{2}$ be improved to $L^{p}$ for some $p = p(t) > 2$? If $\mu$ is $2$-Frostman, then $\mu$ is a bounded compactly supported function, hence $\pi_{\theta} \mu \in L^{\infty}$ for all $\theta \in S^{1}$.

For $t \in (1,2)$, Peres and Schlag \cite{MR1749437} noted that for $p \in [1,2/(2 - t))$ it holds $\pi_{\theta}\mu \in L^{p}$ for $\mathcal{H}^{1}$ almost every $\theta \in S^{1}$. This is a consequence of the Sobolev embedding theorem, and the fact that a.e. projection $\pi_{\theta}\mu$ lies in the fractional Sobolev space $H^{\sigma}$ for $\sigma < (t - 1)/2$. The details, also in higher dimensions, are given in \cite[Section 3.1]{DOV22}. In fact, the following slightly stronger conclusion holds:
\begin{equation}\label{form10} \mathfrak{I}(p,2) := \int_{S^{1}} \|\pi_{\theta}\mu\|_{L^{p}}^{2} \, d\mathcal{H}^{1}(\theta) < \infty, \qquad p \in [1,2/(2 - t)). \end{equation} 
Peres and Schlag asked \cite[Section 9.2(ii)]{MR1749437} whether the integrability threshold $2/(2 - t)$ can be improved. The purpose of this note is to show that the answer is negative:

\begin{thm}\label{thm:main} Let $t \in [1,2)$ and $p > 2/(2 - t)$. Then, there exists a $t$-Frostman Borel measure $\mu$ on $[0,1]^{2}$ such that $\pi_{\theta}\mu \notin L^{p}(\R)$ for every $\theta \in S^{1}$. \end{thm}

Theorem \ref{thm:main} is proved by first establishing in Section \ref{s:deltaDiscretised} a $\delta$-discretised version (Proposition \ref{mainProp}),  and then "stitching" the $\delta$-discretised measures together to generate the final measure $\mu$ by a standard Cantor set procedure. The details are given in Section \ref{s3}. 

\subsection{Related work and open problems} The simple construction given here leaves open the possibility that $\pi_{\theta}\mu \in L^{2/(2 - t)}$ for at least one (or perhaps a.e.) $\theta \in S^{1}$. Another open question remains to determine all pairs $(p,q) \in [1,\infty] \times [1,\infty]$ such that $\mathfrak{I}(p,q) < \infty$ under the $t$-Frostman hypothesis (where $I(p,q)$ is defined as in \eqref{form10}). It was shown in \cite{DOV22} that $\mathfrak{I}(p,p) < \infty$ for $p < (3 - t)/(2 - t)$, and this threshold cannot be improved. More refined information (and mixed norm estimates) was obtained by Liu \cite[Proposition 10.1]{MR4739824}. In \cite{MR4739824} the author also studies the counterpart of the problem for \emph{radial projections}. In \cite{2024arXiv240315784L}, Li and Liu studied a generalised version of the mixed-norm problem, where $\mathcal{H}^{1}$ in \eqref{form10} is replaced by a $\sigma$-Frostman measure on $S^{1}$ (or in general the Grassmannian $G(d,n)$).

\subsection*{Acknowledgements} I would like to thank the reviewer for many useful comments. 

\section{A $\delta$-discretised counterpart of Theorem \ref{thm:main}}\label{s:deltaDiscretised}

In the following, if $A \subset \R^{d}$ and $\delta \in 2^{-\N}$, the notation $\mathcal{D}_{\delta}(A)$ refers to the family of dyadic $\delta$-squares intersecting $A$. In the case $A = \R^{d}$ I abbreviate $\mathcal{D}_{\delta}(\R^{d}) =: \mathcal{D}_{\delta}$.

\begin{definition}[$(\delta,t,C)$-set] Let $C,t \geq 0$. For $\delta,\Delta \in 2^{-\N}$ with $\delta \leq \Delta$, a family $\mathcal{P} \subset \mathcal{D}_{\delta}$ is called a \emph{$(\delta,t,C)$-set between scales $\delta$ and $\Delta$} if
\begin{displaymath} |\mathcal{P} \cap B(x,r)| \leq C(r/\delta)^{t}, \qquad x \in \R^{d}, \, \delta \leq r \leq \Delta. \end{displaymath} 
Here $\mathcal{P} \cap B(x,r) := \{p \in \mathcal{P} : p \cap B(x,r) \neq \emptyset\}$. If $\Delta = 1$, and the constant $C$ is absolute, a $(\delta,t,C)$-set between scales $\delta$ and $\Delta$ will be called a $(\delta,t)$-set.   \end{definition} 

\begin{remark} $(\delta,t)$-sets are commonly known as \emph{Katz-Tao $(\delta,t)$-sets} in reference to \cite{KT01}.  \end{remark} 

Here is the $\delta$-discretised counterpart of Theorem \ref{thm:main}:

\begin{proposition}\label{mainProp} Let $t \in [1,2)$. Then, for all $\delta \in 2^{-\N}$ small enough, there exists a $(\delta,t)$-set $\mathcal{P} \subset \mathcal{D}_{\delta}([0,1]^{2})$ with $|\mathcal{P}| \sim \delta^{-t}$ such that the uniform probability measure $\mu$ on $\cup \mathcal{P}$ satisfies 
\begin{equation}\label{form1} \|\pi_{\theta}\mu\|_{L^{p}} \gtrsim \delta^{1/p - (2 - t)/2}, \qquad \theta \in S^{1}, \, p \geq 1. \end{equation} 
More precisely, for all $\theta \in S^{1}$ there exists an interval $I_{\theta} \subset \R$ of length $|I_{\theta}| \lesssim \delta$, and a set of squares $\mathcal{S}(\theta) \subset \mathcal{P}$ with $|\mathcal{S}(\theta)| \gtrsim \delta^{-t/2}$ such that $\pi_{\theta}(\cup \mathcal{S}(\theta)) \subset I_{\theta}$.
\end{proposition} 

\begin{remark}\label{rem1} The lower bound \eqref{form1} follows from the statement below it, because it implies that $\pi_{\theta}\mu(I_{\theta}) \gtrsim \delta^{t} \cdot |\mathcal{S}(\theta)| \gtrsim \delta^{t/2}$. Therefore, by H\"older's inequality,
\begin{displaymath} \delta^{t/2} \lesssim \int_{I_{\theta}} \pi_{\theta}\mu \leq |I_{\theta}|^{(p - 1)/p}\|\pi_{\theta}\mu\|_{L^{p}} \lesssim \delta^{(p - 1)/p}\|\pi_{\theta}\mu\|_{L^{p}}, \qquad \theta \in S^{1}. \end{displaymath} \end{remark} 

\begin{proof}[Proof of Proposition \ref{mainProp}] The building block in the construction of $\mathcal{P}$ is a \emph{stick}. A stick (or a $t$-stick) $\mathcal{S}$ is a maximal family of $\delta$-squares intersecting a line segment $L \subset \R^{2}$ of length $\delta^{1 - t/2}$. In symbols, $\mathcal{S} = \mathcal{D}_{\delta}(L)$. Note that $|\mathcal{S}| \sim \delta^{-t/2}$.  A line segment $L$ \emph{has direction $\theta_{0} \in S^{1}$} if $L$ is parallel to $\mathrm{span}(\theta_{0})$. A stick $\mathcal{S} = \mathcal{D}_{\delta}(L)$ \emph{has direction $\theta_{0}$} if $L$ has direction $\theta_{0}$.

Regardless of $t \in [1,2)$, the set $\mathcal{P}$ will be constructed as a boundedly overlapping union of $\sim \delta^{-t/2}$ sticks. This ensures that $|\mathcal{P}| \sim \delta^{-t}$.

\subsection{Projections of a single stick} The specific arrangements of sticks in the cases $t \in [\tfrac{4}{3},2)$ and $t \in [1,\tfrac{4}{3})$ are slightly different, see Figures \ref{fig1} and \ref{fig2}, respectively. Before going to the details, we record the following elementary but key property of sticks:
\begin{claim}\label{c1} Let $\theta_{0} \in S^{1}$, and let $\mathcal{S} = \mathcal{D}_{\delta}(L) \subset \mathcal{D}_{\delta}$ be a $t$-stick with direction $\theta_{0}$. Then, 
\begin{displaymath} \diam(\pi_{\theta}(\cup \mathcal{S})) \lesssim \delta, \qquad \theta \in B(\theta_{0},\delta^{t/2}). \end{displaymath} \end{claim}

\begin{proof} Note that $\pi_{\theta_{0}}(L)$ is a singleton (since in this note $\pi_{\theta_{0}}$ stands for the projection to $\theta_{0}^{\perp}$), and recall that $L$ is a segment of length $\delta^{1 - t/2}$. It follows that $\diam(\pi_{\theta}(L)) \lesssim \delta$ for $|\theta - \theta_{0}| \leq \delta^{t/2}$. Since $\pi_{\theta}$ is $1$-Lipschitz, the same inequality persists for $\cup \mathcal{S} = \cup \mathcal{D}_{\delta}(L)$.  \end{proof} 

We deduce the following consequence of Claim \ref{c1}: in order to prove Proposition \ref{mainProp}, it suffices to arrange the $m \sim \delta^{-t/2}$ sticks $\mathcal{S}_{1},\ldots,\mathcal{S}_{m}$ constituting $\mathcal{P}$ in such a way that the $\delta^{t/2}$-neighbourhoods of their directions $\theta_{1},\ldots,\theta_{m}$ cover $S^{1}$. 

\begin{figure}[h!]
\begin{center}
\begin{overpic}[scale = 0.6]{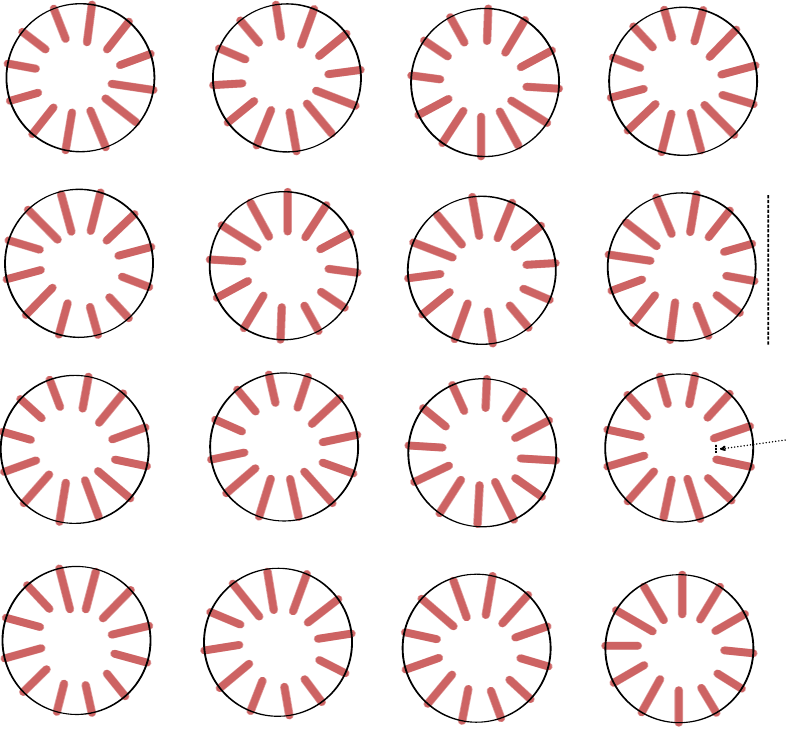}
\put(-4,81){$\mathcal{P}$}
\put(100,36){$\sim \delta^{t - 1}$}
\put(99,57){$20\delta^{1 - t/2}$}
\end{overpic}
\caption{Depiction of $\mathcal{P}$ for $t \in [\tfrac{4}{3},2)$. The red sticks have length $\delta^{1 - t/2}$. As the parameter $t$ increases, the number of "cartwheels" in the picture decreases, and the separation of the sticks decreases. In the extreme case $t = 2$, there would be only one cartwheel of diameter $\sim 1$.}\label{fig1}
\end{center}
\end{figure}

\subsection{Arrangement of sticks: preliminaries}\label{s1} Motivated by the previous explanation, we fix $m \sim \delta^{-t/2}$ \emph{initial sticks} $\mathcal{S}^{0}_{i} := \mathcal{D}_{\delta}(L_{i}^{0})$ such that $0 \in L_{i}^{0}$, and the directions $\theta_{i}$ of $L_{i}^{0}$ are equally spaced on $S^{1}$. Then the desired covering property $S^{1} \subset \bigcup_{i} B(\theta_{i},\delta^{t/2})$ is valid for some $m \sim \delta^{-t/2}$. Eventually we will define
\begin{equation}\label{form11} \mathcal{P} := \bigcup_{i = 1}^{m} \mathcal{S}_{i} := \bigcup_{i = 1}^{m} (\mathcal{S}^{0}_{i} + z_{i}), \end{equation}
where $z_{i} \in \R^{2}$. We will choose the translations $z_{i}$ so that the sticks $\mathcal{S}_{i}$ are disjoint, provided $\delta > 0$ is small enough depending only on $2 - t$ (the initial sticks evidently fail this property), and $\mathcal{P}$ is a $(\delta,t)$-set. There is a difference between the cases $t < \tfrac{4}{3}$ and $t \geq \tfrac{4}{3}$ which we now explain.

Note that every stick has diameter $\leq \delta^{1 - t/2} + 2\delta \leq 2\delta^{1 - t/2}$. Motivated by this, let $B_{1},\ldots,B_{n}$ be a maximal $\delta^{1 - t/2}$-separated collection of discs of radius $10\delta^{1 - t/2}$ contained in $[0,1]^{2}$, see Figure \ref{fig1}. Evidently $n \sim \delta^{t - 2}$. The plan is to place each stick $\mathcal{S}_{i}$ inside one of the discs $B_{j}$, in such a way that each disc $B_{j}$ contains (roughly) the same number of sticks. Since there are $m \sim \delta^{-t/2}$ sticks, this means that each disc $B_{j}$ needs to contain
\begin{displaymath} \sim \frac{m}{n} \sim \delta^{2 - 3t/2} \end{displaymath}
sticks. Now observe that $2 - 3t/2 \leq 0$ if and only if $t \geq 4/3$. Thus, if $t < 4/3$, we have $m < n$, and there are not enough sticks to occupy each disc $B_{1},\ldots,B_{n}$. Indeed, the case $t < 4/3$ leads to a simpler construction which we dispose of immediately.

\begin{figure}[h!]
\begin{center}
\begin{overpic}[scale = 0.6]{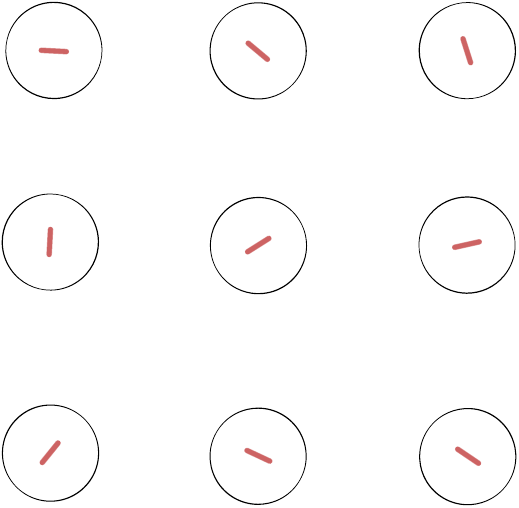}
\put(-6,85){$\mathcal{P}$}
\put(101,49){$20\delta^{1 - t/2}$}
\end{overpic}
\caption{Depiction of $\mathcal{P}$ for $t \in [1,4/3)$. The red sticks have length $\delta^{1 - t/2}$.}\label{fig2}
\end{center}
\end{figure}

\subsection{Case $t \in [1,\tfrac{4}{3})$} Let $\{B_{1},\ldots,B_{m}\} \subset \{B_{1},\ldots,B_{n}\}$ be a maximally separated sub-family. This implies 
\begin{equation}\label{form3} \dist(B_{i},B_{j}) \gtrsim m^{-1/2} \sim \delta^{t/4}, \qquad i \neq j, \end{equation} 
and $\delta^{t/4} \geq \delta^{1 - t/2} \sim \diam(B_{i})$ thanks to the hypothesis $t < 4/3$.

Choose the translations $z_{i} \in \R^{2}$ (as in \eqref{form11}) in such a way that each disc $B_{i}$, $1 \leq i \leq m$, contains exactly one stick $\mathcal{S}_{i} = \mathcal{S}_{i}^{0} + z_{i}$. Since the discs $B_{i}$ are disjoint, the sticks $\mathcal{S}_{i}$ are disjoint, and therefore $|\mathcal{P}| \sim \delta^{-t}$. It remains to check that $\mathcal{P}$ is a $(\delta,t)$-set. 

For $x \in \R^{2}$ and $\delta \leq r \leq \delta^{t/4}$, the disc $B(x,r)$ intersects $\lesssim 1$ discs in $\{B_{1},\ldots,B_{m}\}$ thanks to \eqref{form3}, so (using also the definition of sticks) $|\mathcal{P} \cap B(x,r)| \lesssim r/\delta \leq (r/\delta)^{t}$. For $\delta^{t/4} \leq r \leq 1$, the disc $B(x,r)$ may intersect $\lesssim r^{2}/\delta^{t/2}$ elements from $\{B_{1},\ldots,B_{m}\}$, so 
\begin{displaymath} |\mathcal{P} \cap B(x,r)| \lesssim (\max_{i} |\mathcal{S}_{i}|) \cdot \frac{r^{2}}{\delta^{t/2}} \sim \frac{r^{2}}{\delta^{t}} \leq (r/\delta)^{t}. \end{displaymath}
This completes the proof of Proposition \ref{mainProp} in the cases $t \in [1,\tfrac{4}{3})$.

\subsection{Case $t \in [\tfrac{4}{3},2)$} If $t = \tfrac{4}{3}$, then $m \sim n$, and we may place $O(1)$ sticks inside each disc $B_{1},\ldots,B_{n}$ (if $m < n$, add dummy sticks). After this, the proof is the same as in the case $t < 4/3$. So, we assume that $t > 4/3$, in which case the number $m$ is sticks is "much higher" than the number $n$ of discs.

As already mentioned in Section \ref{s1}, we will place $\sim m/n \sim \delta^{2 - 3t/2}$ sticks $\mathcal{S}_{i}$ inside each disc $B_{j}$. We will next describe, which sticks should be placed inside $B_{j} = B(c_{j},10\delta^{1 - t/2})$, and how they should be placed. See Figure \ref{fig1} for an illustration.

Recall that the directions $\Theta := \{\theta_{1},\ldots,\theta_{m}\}$ of the initial sticks were chosen equally spaced on $S^{1}$. Thus, $|\theta_{k} - \theta_{k + 1}| \sim m^{-1} = \delta^{t/2}$ for $1 \leq k \leq m - 1$. We partition 
\begin{displaymath} \Theta := \bigcup_{j = 1}^{n} \Theta_{j}, \end{displaymath}
where each $\Theta_{j}$ has cardinality $|\Theta_{j}| \sim m/n \sim \delta^{2 - 3t/2}$, and $\Theta_{j}$ remains equally spaced on $S^{1}$. Thus, the separation of neighbouring directions in $\Theta_{j}$ is roughly $\delta^{3t/2 - 2}$. 

For $1 \leq j \leq n$ fixed, we declare that sticks with directions in $\Theta_{j}$ are placed inside a fixed (but arbitrary) disc $B_{j} = B(c_{j},10\delta^{1 - t/2})$. The placement happens as shown in Figure \ref{fig1}: each stick should be placed on such a ray through $c_{j}$ that the direction of the ray coincides with the direction of the stick. Let us be more precise. Let $\mathcal{S}_{i}^{0} = \mathcal{D}_{\delta}(L_{i}^{0})$ be an initial stick with direction $\theta_{i} \in \Theta_{j}$. We choose the translation $z_{i} \in \R^{2}$ so that 
\begin{equation}\label{form4} L_{i} := L_{i}^{0} + z_{i} \subset \bar{B}_{j}, \end{equation}
and $L_{i} \cap \partial B_{j} = \{b_{i}\}$, where $b_{i}$ satisfies $L_{i} \subset [b_{i},c_{j}]$, see again Figure \ref{fig1}. This requirement is fulfilled by either of two "antipodal" points $b_{i} \in \partial B_{j}$, but let us choose the $b_{i}$ satisfying $(b_{i} - c_{j}) \cdot \theta_{i} \geq 0$. We define $\mathcal{S}_{i} := \mathcal{S}_{i}^{0} + z_{i}$. This completes the definition of the translations $z_{i}$, and therefore the definition of $\mathcal{P}$. 

We record here a consequence of the definition of $L_{i}$. Since 
\begin{displaymath} \frac{b_{i} - c_{j}}{10\delta^{1 - t/2}} = \frac{b_{i} - c_{j}}{|b_{i} - c_{j}|} = \theta_{i} \in \Theta_{j},  \end{displaymath} 
it holds $b_{i} \in 10\delta^{1 - t/2}\Theta_{j} + c_{j}$. Since $\Theta_{j}$ is an equally spaced $\delta^{3t/2 - 2}$-separated subset of $S^{1}$, the points $b_{i} \in \partial B_{j}$ form an equally spaced $\sim \delta^{t - 1}$-separated set.

\subsection{The $(\delta,t)$-set property}\label{s:deltaT} We will next show that $\mathcal{P}$ is a $(\delta,t)$-set. First, we consider the case where $x \in \R^{2}$ and $\delta \leq r \leq \tfrac{1}{4}\delta^{1 - t/2}$. Since the discs $B_{j}$ are $\delta^{1 - t/2}$-separated (see Section \ref{s1}), the disc $B(x,r)$ intersects at most one of them, say $B_{j}$. 

To estimate $|\mathcal{P} \cap B(x,r)|$, we may assume that $B(x,r)$ intersects $k \geq 1$ sticks $\mathcal{S}_{i} \subset B_{j}$. Then $B(x,3r)$ intersects $k$ segments $L_{i}$, using the notation \eqref{form4}. We claim that
\begin{equation}\label{form5} k \leq \max\{1,Cr/\delta^{t - 1}\}, \end{equation}
where $C > 0$ is absolute. The moral reason is the following claim:
\begin{claim}\label{c2} If $J \subset \partial B_{j}$ is an arc of length $r \in (0,\diam(\partial B_{j})]$, then 
\begin{displaymath} |\{b_{i} : b_{i} \in J\}| \leq \max\{1,Cr/\delta^{t - 1}\}, \end{displaymath}
where $C > 0$ is absolute.  \end{claim} 
This follows from the equal spacing and $\sim \delta^{t - 1}$-separation of the points $b_{i}$. One has to be a little careful to check that the same estimate persists for the sticks. The enemy is a scenario where all the sticks extend to the centre $c_{j} \in B_{j}$. Then $B(c_{j},\delta)$ would intersect all the $\sim \delta^{2 - 3t/2}$ sticks $\mathcal{S}_{i} \subset B_{j}$, violating \eqref{form5}. The rescue is that the sticks are ten times shorter than the radius of $B_{j}$, so they do not extend close to $c_{j}$.

\begin{figure}[h!]
\begin{center}
\begin{overpic}[scale = 0.9]{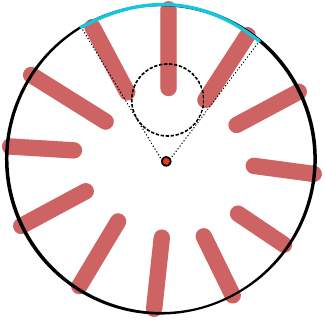}
\put(41,65){\tiny{$B(x,3r)$}}
\put(45,43){$c_{j}$}
\put(46,101){$J$}
\end{overpic}
\caption{Since $B(x,r)$ does not intersect $B(c_{j},\delta^{1 - t/2})$, the "radial projection" $J$ of $B(x,3r)$ to $\partial B_{j}$ has length $|J| \sim r$.}\label{fig3}
\end{center}
\end{figure}

We give the details, see Figure \ref{fig3} for an illustration. Note that since $\mathcal{S}_{i} \cap \partial B_{j} \neq \emptyset$, and $\diam(\mathcal{S}_{i}) \leq 2\delta^{1 - t/2} \leq \tfrac{1}{2}\mathrm{rad}(B_{j})$, the disc $ B(c_{j},5\delta^{1 - t/2})$ has empty intersection with every $\mathcal{S}_{i}$. Therefore, since $B(x,r)$ intersects at least one of the sticks $\mathcal{S}_{i}$ contained in $B_{j}$, it follows from $r \leq \tfrac{1}{4}\delta^{1 - t/2}$ that 
\begin{equation}\label{form6} B(x,3r) \cap B(c_{j},\delta^{1 - t/2}) = \emptyset. \end{equation}
Now, let $J := \{b \in \partial B_{j} : [b,c_{j}] \cap B(x,3r) \neq \emptyset\}$. Alternatively, we may write $J = \rho(B(x,2r))$, where 
\begin{displaymath} \rho(y) := \frac{y - c_{j}}{|y - c_{j}|} \cdot 10\delta^{1 - t/2} \end{displaymath}
is the "radial projection" to $\partial B_{j}$. Note that $\rho$ is $C$-Lipschitz on $B_{j} \, \setminus \, B(c_{j},\delta^{1 - t/2})$, with $C > 0$ absolute. Therefore \eqref{form6} implies $\diam(J) \lesssim r$. Moreover, if $L_{i} \cap B(x,3r) \neq \emptyset$, then also $[b_{i},c_{j}] \cap B(x,3r) \neq \emptyset$, and we infer $b_{i} \in J$. Now \eqref{form5} follows Claim \ref{c2}. In particular, the sticks $\mathcal{S}_{i}$ are disjoint (apply \eqref{form5} with $r \sim \delta \ll \delta^{t - 1}$; here $\delta > 0$ needs to be small enough in terms of $2 - t$).

With \eqref{form5} in hand, and noting that $B(x,r)$ intersects $\lesssim r/\delta$ squares in any fixed stick, we infer (for $x \in \R^{2}$ and $\delta \leq r \leq \tfrac{1}{4}\delta^{1 - t/2}$)
\begin{displaymath} |\mathcal{P} \cap B(x,r)| \lesssim \max\{1,r/\delta^{t - 1}\} \cdot r/\delta = \max\{r/\delta,r^{2}/\delta^{t}\} \leq (r/\delta)^{t}. \end{displaymath}

Consider finally the remaining range $\diam(B_{j}) \sim \tfrac{1}{4}\delta^{1 - t/2} \leq r \leq 1$. Since the discs $B_{j}$ form a maximal $\delta^{1 - t/2}$-separated set of $10\delta^{1 - t/2}$-discs contained in $[0,1]^{2}$,
\begin{displaymath} |\{1 \leq j \leq n : B(x,r) \cap B_{j} \neq \emptyset\}| \lesssim r^{2}/\delta^{2 - t}. \end{displaymath}
Using further that $|\mathcal{P} \cap B_{j}| \lesssim \delta^{-t}/n \sim \delta^{2 - 2t}$ for $1 \leq j \leq n$ fixed, we infer
\begin{displaymath} |\mathcal{P} \cap B(x,r)| \lesssim \delta^{2 - 2t} \cdot (r^{2}/\delta^{2 - t}) = r^{2}/\delta^{t} \leq (r/\delta)^{t}. \end{displaymath}
This completes the proof the of $(\delta,t)$-set property of $\mathcal{P}$, and that of Proposition \ref{mainProp}. \end{proof} 

\section{Proof of Theorem \ref{thm:main}}\label{s3}

We start with the following slightly \emph{ad hoc} definition:

\begin{definition}[Cantor construction]\label{def:CantorConstruction} Given a sequence $\mathcal{P}_{n} \subset \mathcal{D}_{\delta_{n}}([0,1)^{2})$ of families, $n \in \N$, we form a compact set $E \subset [0,1]^{2}$ by a standard \emph{Cantor construction} procedure, described as follows. Let $E_{0} := \cup \mathcal{P}_{0}$. Assume that $E_{n}$, $n \geq 0$, has already been defined, and consists of a union of $\Delta_{n}$-squares $Q_{1},\ldots,Q_{m_{n}}$. To define $E_{n + 1}$, replace each $Q_{k}$, $1 \leq k \leq m_{n}$, by the union of the squares in $\mathcal{P}_{n + 1}$, rescaled inside $Q_{k}$. Set $\Delta_{n + 1} := \Delta_{n} \cdot \delta_{n}$, and define $E_{n + 1}$ as the union of the resulting squares. Finally, set $E := \bigcap_{n} E_{n}$. \end{definition}

The Cantor set $E = E(\{\mathcal{P}_{n}\})$ defined above supports a \emph{canonical probability measure} $\mu$, defined as the weak limit of uniformly distributed (absolutely continuous) probability measures $\mu_{n}$ on the sets $E_{n}$. The measure $\mu$ is moreover related to $\mu_{n}$ as follows:
\begin{equation}\label{form8} \mu(\overline{Q}) \geq \mu_{n}(Q), \qquad Q \in \mathcal{D}_{\Delta_{n}}(E_{n}). \end{equation}
Lemma \ref{DOWLemma} below, taken from \cite[Lemma 5.2]{MR4745881}, gives a condition for $\mu$ to be Frostman. 

\begin{lemma}\label{DOWLemma} Let $E = E(\{\mathcal{P}_{n}\})$ be as in Definition \ref{def:CantorConstruction}, and assume $\{\Delta_{n}\}_{n = 0}^{\infty} \subset 2^{-\N}$ is super-geometrically decaying: for all $\epsilon > 0$ there exists $n_{\epsilon} \in \N$ such that $\Delta_{n} \leq \epsilon \Delta_{n - 1}$ for all $n \geq n_{\epsilon}$.

For each $Q \in \mathcal{D}_{\Delta_{n}}(E)$, assume $|E \cap Q| \geq C^{-1}(\Delta_{n}/\Delta_{n + 1})^{-t}$, and $E \cap Q$ is a $(\Delta_{n + 1},t,C)$-set between scales $\Delta_{n + 1}$ and $\Delta_{n}$, where $C \geq 1$ and $t > 0$ are fixed constants. 

Then, the canonical probability measure on $E$ is $(t - \epsilon)$-Frostman for all $\epsilon > 0$. \end{lemma}

\begin{remark} The conclusion of  \cite[Lemma 5.2]{MR4745881} is actually that $\Hd E \geq t$, but the proof proceeds via showing the $(t - \epsilon)$-Frostman property of $\mu$ (for every $\epsilon > 0$).\end{remark} 

\begin{proof}[Proof of Theorem \ref{thm:main}] Fix $t \in [1,2)$ and $p > t/(2 - t)$ as in Theorem \ref{thm:main}. Fix also $t < u < v < 2$ such that $p > v/(2 - v)$. Let $\{\delta_{n}\}_{n = 0}^{\infty}$ be a super-geometrically decaying sequence, which additionally decays so rapidly that
\begin{equation}\label{form7} (\delta_{0} \cdots \delta_{n - 1})^{2} \geq \delta_{n}^{(v - u)/2}, \qquad n \geq 1.  \end{equation}

 For each $n \geq 0$, let $\mathcal{P}_{n} \subset \mathcal{D}_{\delta_{n}}([0,1)^{2})$ be the $(\delta_{n},u)$-set with $|\mathcal{P}_{n}| \sim \delta_{n}^{-u}$ given by Proposition \ref{mainProp} with index "$u$" in place of "$t$". Let $\{E_{n}\}$ and $\{\Delta_{n}\}$ be the sets and scales associated to $\{\mathcal{P}_{n}\}$ as in Definition \ref{def:CantorConstruction}; note that $\{\Delta_{n}\}_{n = 0}^{\infty} = \{\delta_{0} \cdots \delta_{n}\}_{n = 0}^{\infty}$ is also super-geometrically decaying. Therefore $E$ satisfies the hypotheses of Lemma \ref{DOWLemma} for some absolute constant $C \geq 1$, and the canonical probability $\mu$ on $E$ is $t$-Frostman.

It remains to show that $\pi_{\theta}\mu \notin L^{p}$ for every $\theta \in S^{1}$. Recall that $\mu$ is the (weak) limit of the uniform probabilities $\mu_{n}$ on the sets $E_{n}$.  Recall from Proposition \ref{mainProp} that for every $\theta \in S^{1}$, there exists a set of squares $\mathcal{S}_{n}(\theta) \subset \mathcal{P}_{n}$ with $|\mathcal{S}_{n}(\theta)| \gtrsim \delta_{n}^{-u/2}$ such that 
\begin{equation}\label{form9} \diam(\pi_{\theta}(\cup \mathcal{S}_{n}(\theta))) \lesssim \delta_{n}. \end{equation} 

For $n \geq 1$, the measure $\mu_{n}$ is a convex combination of rescaled and renormalised copies of the uniform probability on $\cup \mathcal{P}_{n}$, denoted $\nu_{n}$. More precisely,
\begin{displaymath} \mu_{n} = \tfrac{1}{m_{n - 1}} \sum_{j = 1}^{m_{n - 1}} T_{Q_{j}}(\nu_{n}), \qquad \{Q_{1},\ldots,Q_{m_{n - 1}}\} := \mathcal{D}_{\Delta_{n - 1}}(E_{n - 1}). \end{displaymath}
Note that $m_{n - 1} \leq \Delta_{m - 1}^{-2} = (\delta_{0} \cdots \delta_{m - 1})^{-2}$. We plan to ignore all the cubes $Q_{j}$ in this decomposition, except one -- for example $Q := Q_{1} \in \mathcal{D}_{\Delta_{n - 1}}(E_{n - 1})$. Here is what we may conclude from \eqref{form9}. Let $T_{Q} \colon [0,1)^{2} \to Q$ be the rescaling map. Then, for each each $\theta \in S^{1}$, there exists a set of squares 
\begin{displaymath} \mathcal{Q}_{n}(\theta) := T_{Q}(\mathcal{S}_{n}(\theta)) \subset \mathcal{D}_{\Delta_{n}}(E_{n}) \end{displaymath} 
with $|\mathcal{Q}_{n}(\theta)| \gtrsim \delta_{n}^{-u/2}$ (in particular $(T_{Q}\nu_{n})(\cup \mathcal{Q}_{n}(\theta)) \gtrsim \delta_{n}^{u/2}$) such that $\diam(\pi_{\theta}(\cup \mathcal{Q}_{n}(\theta))) \lesssim \Delta_{m - 1} \cdot \delta_{n} \leq \delta_{n}$. It follows that there exists an interval $J_{\theta} \subset \R$ of length $|J_{\theta}| \lesssim \delta_{n}$ such that $\pi_{\theta}(\cup \overline{\mathcal{Q}}_{n}(\theta)) \subset J_{\theta}$. Here $\overline{\mathcal{Q}}_{n}(\theta) := \{\overline{Q} : Q \in \mathcal{Q}_{n}(\theta)\}$. 

Now, recall from \eqref{form8} that $\mu(\overline{Q}) \geq \mu_{n}(Q)$ for all $Q \in \mathcal{D}_{\Delta_{n}}(E_{n})$. This implies
\begin{displaymath} (\pi_{\theta}\mu)(J_{\theta}) \geq \mu_{n}(\cup \mathcal{Q}_{n}(\theta)) \geq \tfrac{1}{m_{n - 1}} (T_{Q}\nu_{n})(\cup \mathcal{Q}_{n}(\theta)) \gtrsim (\delta_{0} \cdots \delta_{n - 1})^{2} \cdot \delta_{n}^{u/2}. \end{displaymath}
Applying H\"older's inequality as in Remark \ref{rem1}, and using \eqref{form7}, we deduce
\begin{displaymath} \delta_{n}^{v/2} \leq (\delta_{0} \cdots \delta_{n - 1})^{2} \cdot \delta_{n}^{u/2} \lesssim \delta_{n}^{(p - 1)/p}\|\pi_{\theta}\mu\|_{L^{p}}, \qquad \theta \in S^{1}. \end{displaymath}
Rearranging, we find $\|\pi_{\theta}\mu\|_{L^{p}} \gtrsim \delta_{n}^{1/p - (2 - v)/2}$ for all $\theta \in S^{1}$. Since $p > 2/(2 - v)$, the right hand side diverges as $n \to \infty$, and we conclude $\pi_{\theta}\mu \notin L^{p}$. \end{proof}

\bibliographystyle{plain}
\bibliography{references}

\end{document}